\documentclass[oneside,a4paper,12pt,notitlepage]{article}
\usepackage[left=0.8in, right=0.8in, top=1.5in, bottom=1.5in]{geometry}

\usepackage[T1]{fontenc}
\usepackage[utf8]{inputenc}

\usepackage{lipsum}
\usepackage{lmodern}
\usepackage{amssymb}
\usepackage{amsthm}
\usepackage{bm}
\usepackage{mathtools}

\usepackage{braket}
\usepackage{esint}
\newcommand{\abs}[1]{{\left|#1\right|}}
\newcommand{\norma}[1]{{\left\Vert#1\right\Vert}}

\usepackage{booktabs}
\usepackage{graphicx}
\usepackage{tikz}
\usetikzlibrary{patterns, math, intersections}
\usepackage{multicol}
\usepackage{caption}
\usepackage{enumerate}
\usepackage[skins,theorems]{tcolorbox}
\tcbset{highlight math style={enhanced,
		colframe=black,colback=white,arc=0pt,boxrule=1pt}}
\def\XXint#1#2#3{{\setbox0=\hbox{$#1{#2#3}{\int}$}
    \vcenter{\hbox{$#2#3$}}\kern-.5\wd0}}

\theoremstyle{definition}
\newtheorem{definizione}{Definition}[section]
\theoremstyle{plain}
\newtheorem{teorema}{Theorem}[section]

\newtheorem{lemma}[teorema]{Lemma}

\newtheorem{corollario}[teorema]{Corollary}
\theoremstyle{definition}

\newtheorem{oss}[teorema]{Remark}

\DeclareMathOperator{\R}{\mathbb{R}}

\DeclareMathOperator*{\esssup}{\text{ess sup}}
\DeclareMathOperator*{\essinf}{\text{ess inf}}
\DeclareMathOperator*{\essosc}{\text{ess osc}}

\makeatletter
\newcommand{\myfootnote}[2]{\begingroup
	\def\@makefnmark{}%
	\addtocounter{footnote}{-1}%
	\footnote{\textbf{#1} #2}%
	\endgroup}
\makeatother

\newcommand{\parzder}[2]{\frac{\partial {#1}}{\partial {#2}}}
\newcommand{\asso}{\star}
\newcommand{\convdeb}{\rightharpoonup}
\newcommand{\eps}{\varepsilon}
\newcommand{\meno}{\setminus}
\usepackage{comment}

\newcommand{\lnorma}[1]{{\left\vert\kern-0.25ex\left\vert\kern-0.25ex\left\vert #1 \right\vert\kern-0.25ex\right\vert\kern-0.25ex\right\vert}}
\theoremstyle{definition}
\newtheorem{open problem}{Open Problem}

\numberwithin{equation}{section}

\usepackage{todonotes}

\usepackage[style = alphabetic, maxbibnames=99, maxcitenames = 99, maxbibnames = 99, maxalphanames = 99]{biblatex}
\AtEveryBibitem{%
	\clearfield{issn} % Remove issn
	\clearfield{doi}
	\ifentrytype{online}{}{% Remove url except for @online
		\clearfield{url}
	}
}

\usepackage{hyperref}
\hypersetup{linktoc=none, bookmarksnumbered, colorlinks=true, linkcolor=red}

\title{On the symmetric rearrangement of the gradient of a Sobolev function}
\author{Vincenzo Amato, Andrea Gentile}
\date{\today}
\begin{document}
	\maketitle

\begin{abstract}
    In this paper, we generalize a classical comparison result for solutions to Hamilton-Jacobi equations with Dirichlet boundary conditions, to solutions to Hamilton-Jacobi equations with non-zero boundary trace.

    As a consequence, we prove the isoperimetric inequality for the torsional rigidity (with Robin boundary conditions) and for other functionals involving such boundary conditions. 

\end{abstract}

\section{Introduction}
Let $\Omega$ be a bounded, open and Lipschitz set and let $u \in W^{1,p}(\Omega)$, for some $p \geq 1$, be a non-negative function.

In this paper, we deal with the problem of comparing a function $u \in W^{1,p}(\Omega)$ with a radial function having the modulus of the gradient equi-rearranged with $\abs{\nabla u}$.
Hence, we aim to extend the results contains in  and Nunziante \cite{GN} to a more general setting. 

Throughout this article, $\abs{\cdot}$ will denote both the $n$-dimensional Lebesgue measure and the $(n-1)$-dimensional Hausdorff measure, the meaning will be clear by the context.

If $A$ is a bounded and open set with the same measure as $\Omega$, we say that a function $f^\asso \in L^p(A)$ is equi-rearranged to $f \in L^p(\Omega)$ if they have the same distribution function, i.e.
\begin{definizione}
	Let $f: \Omega \to \R$ be a measurable function, the \emph{distribution function} of $f$ is the function $\mu_f : [0,+\infty[\, \to [0, +\infty[$ defined by
	\[
	\mu_f(t)= \abs{\Set{x \in \Omega \, :\,  \abs{f(x)} > t}}.
	\]
\end{definizione}

In order to state our results, we recall some definitions
\begin{definizione}
\label{rearrangements}
Let $f: \Omega \to \R$ be a measurable function:
    \begin{itemize}
    \item  the \emph{decreasing rearrangement} of $f$, denoted by $f^\ast$, is the distribution function of $\mu_f$. Moreover, we can write
    \[
    f^\ast(s)= \inf \{ t \geq 0 \, |\, \mu_f(t) < s\};
    \]
    \item  the \emph{increasing rearrangement} of $f$ is defined as
    \[
    f_\ast(s)= f^\ast(\abs{\Omega}-s);
    \]
    \item the \emph{spherically symmetric decreasing rearrangement} of $f$, defined in $\Omega^\sharp$ i.e. the ball centered at the origin with the same measure as $\Omega$, is the function
    \[
    f^\sharp(x) = f^\ast(\omega_n \abs{x}^n),
    \]
    where $\omega_n$ is the measure of the $n$-dimensional unit-ball of $\R^n$;
    \item the \emph{spherically symmetric increasing rearrangement} of $f$, defined in $\Omega^\sharp$, is
    \[
    f_\sharp(x) = f_\ast(\omega_n \abs{x}^n).
    \]
    \end{itemize}
\end{definizione}
Clearly, we can construct several rearrangements of a given function $f$, but the one we will refer to is the spherically symmetric increasing rearrangement defined in $\Omega^\sharp$.

The starting point of our work, and many others, is \cite[Theorem 2.2]{GN}

\begin{teorema}
	\label{Giarrusso_Nunziante}
	Let $p \geq 1$, $f \colon \Omega \to \R$, $H \colon \R^n \to \R$ be measurable non-negative functions and let $K \colon [0,+\infty) \to [0,+\infty)$ be a strictly increasing real-valued function such that
	\[
	0 \leq K(\abs{y}) \leq H(y) \qquad \forall y \in \R^n \qquad \text{ and } K^{-1}(f) \in L^p(\Omega)
	\]
	Let $v \in W_0^{1,p}(\Omega)$ be a function that satisfy
	\[
	\begin{cases}
		H(\nabla v) = f(x) &\text{a.e. in }\Omega \\
		v = 0 &\text{on } \partial \Omega
	\end{cases}
	\]
	then, denoting with $\overline{v}$ the unique decreasing spherically symmetric solution to
	\[
	\begin{cases}
		K(\abs{\nabla \overline{v}}) = f_{\sharp}(x) & \text{a.e. in } \Omega^{\sharp} \\
		\overline{v} = 0 & \text{on } \partial \Omega^{\sharp}
	\end{cases}
	\]
	it holds
	\begin{equation}
		\label{eq_Giarrusso_Nunziante}
		\norma{v}_{L^1(\Omega)} \leq \norma{\overline{v}}_{L^1(\Omega^{\sharp})}
	\end{equation}
\end{teorema}

They give also a similar result for the spherically symmetric decreasing rearrangement of the gradient, with an $L^\infty$ comparison.

In recent decades, many authors studied this kind of problems, in particular in \cite{ALT} Alvino, Lions and Trombetti proved the existence of a spherically symmetric rearrangement of the gradient of $v$ which gives a $L^q$ comparison as in \eqref{eq_Giarrusso_Nunziante} for a fixed $q$.

Moreover, Cianchi in \cite{Cia} gives a characterization of such rearrangement; clearly, the rearrangement found by Cianchi is different both from the spherically symmetric increasing and decreasing rearrangement if $q \in (1, \infty)$.

Furthermore, in \cite{Ferone_Posteraro} and \cite{Ferone_Posteraro_Volpicelli} the authors studied the optimization of the norm of a Sobolev function in the class of functions with fixed rearrangement of the gradient.

Incidentally, let us mention that the case where the $L^{q, 1}$ Lorentz norm, see Section \ref{Section_2} for its definition, takes the place of the $L^q$ norm in \eqref{eq_Giarrusso_Nunziante} has been studied in \cite{Ta6}. In particular, he stated the following
\begin{teorema}
    Let $u$ be a real-valued function defined in $\R^n$. Suppose $u$ is nice enough - e.g. Lipschitz continuous -  and the support of $u$ has finite measure. Let $M$ and $V$ denote the distribution function of $\abs{\nabla u}$ and the measure of the support of $u$, respectively.
    
    Let $v$ the real-valued function defined in $\R^n$ that satisfies the following conditions:
    \begin{enumerate}
        \item $\abs{\nabla v}$ is a rearrangement of $\abs{\nabla u}$;
        \item the support of $v$ has the same measure of the support of $u$;
        \item $v$ is radially decreasing and $\abs{\nabla v}$ is radially increasing.
    \end{enumerate}
    Then 
    \[
    \norma{u}_{L^{p,1}(\Omega)} \leq \norma{v}_{L^{p,1}(\Omega^{\sharp})} \quad \text{ if } n=1 \text{ or } 0 < p \leq \frac{n}{n-1},
    \]
    furthermore
    \[
    \norma{v}_{L^{p,1}(\Omega^{\sharp})} = \frac{p^2}{ \omega_n^{\frac{1}{n}}(n+p)} \int_0^\infty \left[V^{\frac{1}{p} + \frac{1}{n}} - (V-M(t))^{\frac{1}{p} + \frac{1}{n}}\right]\, dt.
    \]
\end{teorema}

On the other hand, the problem of studying the rearrangement of the Laplacian has been widely studied by several authors. The bibliography is extensive; for the sake of completeness, let us recall some of the works: \cite{Ta5} for the Dirichlet boundary conditions,  \cite{ACNT_sharp_estimates,ANT,AGM} for the Robin conditions.

As we already said, we focus on the case in which the functions do not vanish on the boundary. Our main theorem is the following:

\begin{teorema}
    \label{Teorema_che_scriveremo}
    Let $\Omega \subset \R^n$ be a bounded, open and Lipschitz set and let $u \in W^{1,p}(\Omega)$ be a non-negative function. If we denote with $\Omega^{\sharp}$ the ball centered at the origin with same measure as $\Omega$, then there exists a non-negative function $u^{\asso} \in W^{1,p}(\Omega^{\sharp})$ that satisfies
    \begin{equation}
        \label{eq_che_risolve_u_picche}
        \begin{cases}
            \lvert \nabla u^{\asso} \rvert = \abs{\nabla u}_{\sharp}(x) & \text{a.e. in }\Omega^{\sharp} \\[1ex]
            u^\asso = \cfrac{\displaystyle{ \int_{\partial \Omega}  u \, d \mathcal{H}^{n-1}} }{\displaystyle{ \lvert \partial \Omega^{\sharp} \rvert }} &\text{ on } \partial \Omega^{\sharp}.
        \end{cases}
    \end{equation}
    and such that
    \begin{align}
        \label{norma_L1}
        \norma{u}_{L^1(\Omega)} \leq \norma{u^{\asso}}_{L^1(\Omega^{\sharp})}.
    \end{align}
\end{teorema}

\begin{oss}
\label{remark_holder}
    By the explicit expression of $u^*$ on the boundary and the H\"{o}lder inequality, we can estimate the $L^p$ norm of the trace:
    \begin{equation}
        \label{norma_traccia_Lp}
        \lvert \partial \Omega^{\sharp} \rvert^{p-1} \int_{\partial \Omega^{\sharp}} (u^{\asso})^p \, d\mathcal{H}^{n-1} = \left(\int_{\partial \Omega} u \, d\mathcal{H}^{n-1}\right)^p \leq \abs{\partial \Omega}^{p-1}\int_{\partial \Omega} u^p \, d\mathcal{H}^{n-1} \qquad \forall p \geq 1.
    \end{equation}
\end{oss}

This result allows us to compare solutions to PDE with Robin boundary conditions with solutions to their symmetrized.

Precisely we are able to compare solutions to
\begin{equation}
    \label{eq_soluzione_debole_Omega_intro}
    \begin{cases} 
    -\Delta u = 1 &\text{in } \Omega \\[1ex]
    \displaystyle{\parzder{u}{\nu} + \beta \abs{\partial \Omega } \, u = 0} &\text{on } \partial \Omega
    \end{cases}
\end{equation}
with the solution to
\begin{equation}
    \label{eq_soluzione_debole_Omega_sharp_intro}
    \begin{cases} 
    -\Delta v = 1 &\text{in } \Omega^{\sharp} \\[1ex]
    \displaystyle{\parzder{v}{\nu} + \beta \lvert \partial \Omega^{\sharp} \rvert \, v = 0} &\text{on } \partial \Omega^{\sharp}
    \end{cases}
\end{equation}
In particular we get
\begin{corollario}
    \label{corollario_torsione_pesata}
    Let $\beta>0,$ let $\Omega \subset \R^n$ be a bounded, open and Lipschitz set. If we denote with $\Omega^{\sharp}$ the ball centered at the origin with same measure as $\Omega$, it holds
    \begin{equation}
        T(\Omega,\beta) \geq T(\Omega^{\sharp},\beta) 
    \end{equation}
\end{corollario}
where
\begin{equation}
   T(\Omega,\beta)  = \inf_{w \in W^{1,2}(\Omega)} \cfrac{ \displaystyle{  \int_{\Omega} \abs{\nabla w}^2 \, dx + \beta \abs{\partial \Omega}  \, \int_{\partial \Omega} w^2 \, d\mathcal{H}^{n-1} } }{\displaystyle{ \biggl(\int_{\Omega} w \, dx \biggr)^2 }} \qquad \text{for } w \in W^{1,2}(\Omega).
\end{equation}

The paper is organized as follows. In Section \ref{Section_2} we recall some basic notions, definitions and
classical results and we prove Theorem \ref{Teorema_che_scriveremo}. Eventually, Section \ref{Section_4} is dedicated to the application to the Robin torsional rigidity and in Section \ref{Section_5} we get a comparison between Lorentz norm of $u$ and $u^{\asso}$.

\section{Notations, Preliminaries and proof of the main result}
\label{Section_2}
Observe that obviously $\forall p \geq 1$

\[
    \displaystyle{\norma{f}_{L^p(\Omega)}=\norma{f^*}_{L^p([0, \abs{\Omega}])}=\lVert{f^\sharp}\rVert_{L^p(\Omega^\sharp)}=\norma{f_*}_{L^p([0, \abs{\Omega}])}=\lVert{f_\sharp}\rVert_{L^p(\Omega^\sharp)}},
\]
moreover, the Hardy-Littlewood inequalities hold true:
\begin{equation*}
    \int_{\Omega} \abs{f(x)g(x)} \, dx \le \int_{0}^{\abs{\Omega}} f^*(s) g^*(s) \, ds= \int_{\Omega^\sharp} f^\sharp(x) g^\sharp(x) \, dx,
\end{equation*}
\begin{equation*}
    \int_{\Omega^\sharp} f^\sharp(x) g_\sharp(x) \, dx=\int_{0}^{\abs{\Omega}} f^*(s) g_*(s) \, ds \leq \int_{\Omega} \abs{f(x)g(x)} \, dx .
\end{equation*}
Finally, the operator which assigns to a function its symmetric decreasing rearrangement is a contraction in $L^p$ , see (\cite{CP}), i.e.
\begin{equation}
    \label{eq_riarr_diminuiscono_distanza_Lp}
    \norma{f^*-g^*}_{L^p([0,\abs{\Omega}])} \leq \norma{f-g}_{L^p(\Omega)}
\end{equation}
One can find more results and details about rearrangements for instance in \cite{HLP} and in \cite{Ta6}.

Other powerful tools are the pseudo-rearrangements. Let $u \in W^{1,p}(\Omega)$ and let $f \in L^1(\Omega)$, as in \cite{AT} $\forall s \in [0,\abs{\Omega}]$, there exists a subset $D(s) \subseteq \Omega$ such that 
\begin{enumerate}

    \item $\abs{D(s)}=s$;
    
    \item $D(s_1) \subseteq D(s_2)$ if $s_1<s_2$;
    
    \item $D(s) = \Set{x \in \Omega \, | \, \abs{u(x)}>t}$ if $s=\mu(t)$.

\end{enumerate} 

So the function
\[
\int_{D(s)}f(x) \, dx
\]
is absolutely continuous, therefore it exists a function $F$ such that
\begin{equation}
\label{assol_cont}
    \int_0^s F(t) \, dt = \int_{D(s)} f(x) \, dx
\end{equation}
We will use the following propriety (\cite[Lemma 2.2]{AT})
\begin{lemma}
    \label{lemma_Alvino_Trombetti}
    Let $f \in L^p$ for $p>1$ and let $D(s)$ be a family described above. If $F$ is defined as in \eqref{assol_cont}, then there exists a sequence $\Set{ F_k }$ such that $F_k$ has the same rearrangement as $f$ and
    \[
    F_k \convdeb F \qquad \text{in } L^p([0,\abs{\Omega}])
    \]
    If $f \in L^1$ it follows that
    \[
    \lim_k \int_0^{\abs{\Omega}} F_k(s) g(s) \, ds = \int_0^{\abs{\Omega}} F(s) g(s) \, ds
    \]
    for each function $g \in BV([0,\abs{\Omega}])$.
\end{lemma}

Moreover, for sake of completeness, we will recall the definition of the Lorentz norm.

\begin{definizione}
    \label{def_spazi_Lorentz}
    Let $\Omega \subseteq \R^n$ a measurable set, $0<p<+\infty$ and $0<q<+\infty$. Then a function $g$ belongs to the Lorentz space $L^{p,q}(\Omega)$ if 
    \begin{equation}
        \label{eq_norma_Lorentz}
        \norma{g}_{L^{p,q}(\Omega)} = \biggl( \int_0^{+\infty} \bigl[t^{\frac{1}{p}} g^*(t) \bigr]^q \frac{dt}{t} \biggr)^{\frac{1}{q}} < + \infty.
    \end{equation}
\end{definizione}
Let us notice that for $p=q$ the Lorentz space $L^{p,p}(\Omega)$ coincides with the Lebesgue space $L^p(\Omega)$ by the Cavalieri's principle.

Let us now prove the main  Theorem.

\begin{proof}[Proof of Theorem \ref{Teorema_che_scriveremo}]
Let us consider $\varepsilon$ and $\delta:=\delta_\varepsilon$ and the sets
\begin{equation}
    \label{def_Omega_eps}
    \begin{aligned}
        \Omega_{\varepsilon} &= \Set{x \in \R^n | d(x, \Omega) < \varepsilon} \qquad &  \Sigma_\varepsilon &= \Omega_\varepsilon \setminus \Omega,\\
        \Omega^{\sharp}_\varepsilon & = \Set{x \in \R^n | d(x, \Omega^{\sharp}) < \delta} \qquad &  \Sigma^\sharp_\varepsilon &= \Omega^\sharp_\varepsilon \setminus \Omega^\sharp,\\
        \abs{\Omega_\varepsilon} &=  \lvert \Omega^{\sharp}_\varepsilon \rvert \qquad & \abs{\Sigma_{\eps}} &= \lvert \Sigma_{\eps}^{\sharp} \rvert,
    \end{aligned}
\end{equation}
where, since $\abs{\Sigma_{\eps}}/\varepsilon \to\abs{\partial\Omega} $ and $ \lvert \Sigma_{\eps}^{\sharp}\rvert/ \delta \to \lvert \partial\Omega^\sharp \rvert$ as $\varepsilon \to 0$, we have
\[
\lim_{\varepsilon \to 0} \, \frac{\delta}{\varepsilon} = \frac{\abs{\partial \Omega}}{\abs{\partial \Omega^\sharp}}.
\]
Let $d(\cdot, \Omega)$ defined as follows: 
$$d(x,\Omega):=\inf_{y\in \Omega}\abs{x-y}.$$
Then we divide the proof into four steps.
    
    \begin{enumerate}
        
        \item[\textbf{Step 1}] First of all we assume $\Omega $ with $ C^{1,\alpha}$ boundary, $u \in W^{1,\infty}(\Omega)$ and $u \geq \sigma >0$ in $\Omega$.
        
        So we can consider the following "linear" extension of $u$, $u_\varepsilon$ in $\Omega_\varepsilon$
        \[
        u_{\eps}(x) = u \bigl( p(x) \bigr) \biggl( 1-\frac{d(x,\partial \Omega)}{\eps} \biggr) \qquad \forall x \in \Omega_\varepsilon\setminus\Omega,
        \]
        where $p(x)$ is the projection of $x$ on $\partial\Omega$ (for $\varepsilon$ sufficiently small, this definition is well posed since $\Omega$ is smooth, see \cite{GT}). The function $u_\varepsilon$, has the following properties:
    
        \begin{enumerate}
            \item $\displaystyle{u_\varepsilon \rvert_{\Omega} = u}$,
            
            \item $\displaystyle{u_\varepsilon=0}$ on $\partial \Omega_\varepsilon$,
            
            \item\label{pro-gra_1.1} $\displaystyle{\norma{\nabla u_\varepsilon}_{L^{\infty}(\Omega)} \leq \abs{\nabla u_\varepsilon}(y)} $ $\forall y \in \Sigma_\varepsilon$ for $\varepsilon$ sufficiently small,
            
            \item $\displaystyle{ \lim_{\varepsilon \to 0^+} \int_{\Sigma_\varepsilon} \abs{\nabla u_\varepsilon}\, dx = \int_{\partial\Omega} u \, d\mathcal{H}^{n-1}.}$
        \end{enumerate}
        Properties $(a)$ and $(b)$ follow immediately by the definition of $u_{\eps}$, while $(c)$ is a consequence of the regularity of $u$. Property $(d)$ can be obtained by an easy calculation, indeed
        \[
        \nabla u_{\eps} (x) = \nabla \bigl( u(p(x)) \bigr) \biggl[ 1-\frac{d(x,\partial \Omega)}{\eps} \biggr] - u \bigl( p(x) \bigr) \frac{\nabla d(x,\partial \Omega)}{\eps}
        \]
        For the first term, we can notice
        \[
        \int_{\Sigma_{\eps}} \bigl \lvert \nabla \bigl(u (p(x)) \bigr) \bigr \rvert \biggl[ 1-\frac{d(x,\partial \Omega)}{\eps} \biggr] \, dx  \leq L \int_{\Sigma_{\eps}} \, dx = L \abs{\Sigma_{\eps}}
        \]
        where $L$ is the $L^{\infty}$ norm of $\nabla u(p(x))$. Now we deal with the second term and, keeping in mind that $\abs{\nabla d} = 1$ and using coarea formula, we have
        \begin{align*}
            \lim_{\eps \to 0^+} \int_{\Sigma_{\eps}} \abs{\nabla u_{\eps}} \, dx & = \lim_{\eps \to 0^+} \frac{1}{\eps} \int_{\Sigma_{\eps}} u(p(x)) \, dx  = \lim_{\eps \to 0^+} \int_0^{\eps} \, dt \int_{\Gamma_t} (u\circ p) \, d \mathcal{H}^{n-1}
        \end{align*}
        where $\Gamma_t = \Set{x \in \Sigma_{\eps} \, | \, d(x,\partial \Omega) = \eps}$. By continuity of $u$ and Lebesgue differentation theorem we get
        \[
        \lim_{\eps \to 0^+} \frac{1}{\eps} \int_0^{\eps} \, dt \int_{\Gamma_t} u\circ p \, d \mathcal{H}^{n-1} = \int_{\Gamma_0} (u \circ p) \, d\mathcal{H}^{n-1} = \int_{\partial \Omega} u \, d \mathcal{H}^{n-1}
        \]
        that proves property $(d)$.
        
        For every $\varepsilon>0$, we consider the following problem
        \begin{equation}
    	    \label{eps_gn}
    	    \begin{cases}
    	    	\abs{\nabla v_\varepsilon} (x) = \abs{\nabla u_\varepsilon}_{\sharp} (x) &\text{ in } \Omega_\varepsilon^{\sharp} \\
        		v_\varepsilon = 0 &\text{ on } \partial \Omega_\varepsilon^{\sharp}
    	    \end{cases}
        \end{equation}
        and by Theorem \ref{Giarrusso_Nunziante} it holds
        \begin{equation}
            \label{Giarrusso_Nunziante_con_u_eps}
    	    \norma{u_\varepsilon}_{L^1(\Omega_\varepsilon)} \leq  \norma{v_\varepsilon}_{L^1(\Omega^\sharp_\varepsilon)}.
        \end{equation}
        Moreover it exists $\bar{\varepsilon}$ such that for every $\varepsilon \leq \bar{\varepsilon}$
        \begin{equation}
            \label{gradienti_uguali_dentro}
        	\abs{\nabla v_\varepsilon} (x) = \abs{\nabla u_{\eps}}_{\sharp}(x) = \abs{\nabla u}_{\sharp} (x)  \qquad \forall x \in \Omega^\sharp.
        \end{equation}
        We can see $u_\varepsilon$ as a $W^{1,1}(\Omega_{\bar{\varepsilon}})$ function and we have
        \begin{equation}
            \label{gradiente_limitato}
            \begin{split}
            \int_{\Omega_{\bar{\varepsilon}}^\sharp}\abs{\nabla v_\varepsilon} =\int_{\Omega_{\bar{\varepsilon}}}\abs{\nabla u_\varepsilon} &=  \int_{\Omega}\abs{\nabla u} + \int_{\Sigma_\varepsilon}\abs{\nabla u_\varepsilon} 
            \leq \norma{\nabla u}_{L^1(\Omega)} + 2 \norma{u}_{L^1(\partial \Omega)}.
            \end{split}
        \end{equation}
        by property $(d)$.

        \noindent Finally, by Poincarè and \eqref{gradiente_limitato}, there exists a constant $0<C=C(n,\Omega)$ such that
        \begin{equation*}
            \norma{v_\varepsilon}_{W^{1,1}(\Omega_{\bar{\varepsilon}}^\sharp)} \leq C \norma{\nabla v_\varepsilon}_{L^{1}(\Omega_{\bar{\varepsilon})}} \leq  C(n,\Omega) \norma{u}_{W^{1,1}(\Omega)}.
        \end{equation*}

        Therefore, up to a subsequence, there exists a limit function $u^\asso \in BV(\Omega_{\bar{\varepsilon}}^\sharp)$ such that (\cite[Proposition 3.13]{AFP})
        \begin{equation*}
            v_\varepsilon \to u^\asso \text{ in } L^1(\Omega_{\bar{\varepsilon}}^\sharp) \qquad \nabla v_{\varepsilon} \overset{*}{\convdeb} \nabla u^{\asso} \text{ in } \Omega
        \end{equation*}
        namely
        \[
            \lim_{\eps \to 0} \int_{\Omega_{\bar{\varepsilon}}^{\sharp}} \varphi \, d \nabla v_{\eps}  = \int_{\Omega_{\bar{\varepsilon}}^{\sharp}} \varphi \, d \nabla u^{\asso} \qquad \forall \varphi \in C_0(\Omega,\R^n)
        \]
        Our aim is to show that $u^\asso$ satisfies properties \eqref{eq_che_risolve_u_picche}, \eqref{norma_L1} and \eqref{norma_traccia_Lp}.

        Concerning \eqref{eq_che_risolve_u_picche} then $\lvert \nabla u^{\asso} \rvert = \abs{\nabla u}_{\sharp}$ follows from \eqref{gradienti_uguali_dentro}.
        
        To find the value of $u^\star$ at the boundary, we observe that, from \eqref{eps_gn} and \eqref{gradienti_uguali_dentro}, we have 
        \begin{equation*}
            \int_{\Sigma_{\varepsilon}} \abs{\nabla u_\varepsilon} = \int_{\Sigma^\sharp_{\varepsilon}} \abs{\nabla v_\varepsilon}.
        \end{equation*}
        Now, for $t>0$ setting $\Gamma_t= \Set{d(x, \Omega) = t}$, $\Gamma_t^{\sharp} = \{d(x, \Omega^{\sharp}) = t \}$, $r=\displaystyle{\biggl( \frac{\abs{\Omega}}{\omega_n} \biggr)^{\frac{1}{n}}}$ and recalling that $v_{\eps}$ is radially symmetric we have
        \begin{gather*}
            \int_{\Sigma^\sharp_{\varepsilon}} \abs{\nabla v_\varepsilon} =  \int_r^{r+\delta} \int_{\Gamma^\sharp_t}  \abs{\nabla v_\varepsilon} \, d \mathcal{H}^{n-1} \, dt = \lvert \Gamma^\sharp_t \rvert \int_r^{r+\delta} - v'_\varepsilon \lvert \Gamma^\sharp_t \rvert \, dt = \lvert \Gamma^\sharp_t \rvert \, v_\varepsilon(r).
        \end{gather*}
        Therefore by monotonicity of $\lvert \Gamma_t^{\sharp} \rvert$ we have
        \[
        \lvert \Gamma^\sharp_r \rvert v_{\eps}(r) \leq \int_r^{r+\delta} \bigl( -v_{\eps}'(t) \lvert \Gamma_t^{\sharp} \rvert \bigr) \, dt \leq \lvert \Gamma_{r+\delta}^{\sharp} \rvert v_{\eps}(r)
        \]
        and since
        \[
        \lvert \Gamma^\sharp_r \rvert v_{\eps}(r) = \int_{\partial \Omega^{\sharp}} v_{\eps} \, d\mathcal{H}^{n-1}
        \]
        using the fact that $v_{\eps} \to v$ in $L^{1}(\Omega)$, $\nabla v_{\eps} = \nabla u$ in $\Omega$ and the continuity embedding of $W^{1,1}(\Omega)$ in $L^1(\Omega)$, in the end we have
        \[
        \int_{\Sigma^\sharp_{\varepsilon}} \abs{\nabla v_\varepsilon} \to \int_{\partial \Omega^{\sharp}} u^{\asso} \, d\mathcal{H}^{n-1}.
        \]
        Using property $(d)$ we obtain
        \begin{equation*}
            \int_{\partial \Omega}  u \, d \mathcal{H}^{n-1}= \int_{\partial \Omega^\sharp}   u^\asso  \, d \mathcal{H}^{n-1}.
        \end{equation*}
        In the end we have that for $u^\asso$ it holds
        \begin{equation}
            \begin{cases}
            \abs{\nabla u^\asso} = \abs{\nabla u}_{\sharp} &\text{ in } \Omega^{\sharp} \\
            u^\asso = \cfrac{\displaystyle{ \int_{\partial \Omega}  u \, d \mathcal{H}^{n-1}}}{\displaystyle{\lvert \partial \Omega^{\sharp} \rvert}} &\text{ on } \partial \Omega^{\sharp}.
            \end{cases}
        \end{equation}
        that proves \eqref{eq_che_risolve_u_picche}.
        
        Furthermore by
        \begin{equation*}
            \norma{u_\varepsilon}_{L^1(D)} \to \norma{u}_{L^1(D)} \quad \text{ and }  \quad \norma{v_\varepsilon}_{L^1(D^\sharp)} \to \lVert{u^\asso}\rVert_{L^1(D^\sharp)}.
        \end{equation*}
        we can pass to the limit $\varepsilon \to 0$ in \eqref{Giarrusso_Nunziante_con_u_eps} and we get

        \begin{equation*}
            \norma{u}_{L^1(\Omega)} \leq  \lVert{u^\asso}\rVert_{L^1(\Omega^\sharp)}.
        \end{equation*}
        that proves \eqref{norma_L1}.

        \item[\textbf{Step 2}]
        Now we remove the extra-assumption $u \geq \delta >0$ defining
        \begin{equation*}
            u_\sigma := u+ \sigma.
        \end{equation*}
    
        Then $u_{\sigma}$ is strictly positive in $\Omega$ and we can apply the previous result: there exists a function $v_\sigma$ in $\Omega^\sharp$ such that
        \begin{equation*}
            \begin{cases}
                \abs{\nabla v_\sigma} =
                \abs{\nabla u_\sigma}_{\sharp}=\abs{\nabla u}_{\sharp} &\text{ a.e. in } \Omega^{\sharp} \\
                v_\sigma =
                \cfrac{\displaystyle{ \int_{\partial \Omega}  u_\sigma \, d \mathcal{H}^{n-1}}}{\displaystyle{\lvert \partial \Omega^{\sharp} \rvert }} =
                \cfrac{\displaystyle{ \int_{\partial \Omega}  u \, d \mathcal{H}^{n-1}}}{\displaystyle{ \lvert \partial \Omega^{\sharp} \rvert }} + \sigma \frac{\displaystyle{\abs{\partial \Omega}}}{\displaystyle{\lvert \partial \Omega^\sharp \rvert}} &\text{ on } \partial \Omega^{\sharp},
            \end{cases}
        \end{equation*}
        and
        \begin{equation}
            \label{eq_per_v_sigma_e_u_sigma}
            \norma{u_\sigma}_{L^1(\Omega)} \leq  \lVert{v_\sigma}\rVert_{L^1(\Omega^\sharp)},
        \end{equation}
        If we define 
        \begin{equation*}
            u^\asso:= v_\sigma - \sigma \frac{\displaystyle{\abs{\partial \Omega}}}{\displaystyle{\lvert \partial \Omega^\sharp \rvert}},
        \end{equation*}
        then $u^{\asso}$ solves
        \begin{equation}
            \begin{cases}
            \abs{\nabla u^\asso} = \abs{\nabla u}_{\sharp} &\text{ in } \Omega^{\sharp} \\
            u^\asso = \cfrac{\displaystyle{ \int_{\partial \Omega}  u \, d \mathcal{H}^{n-1}}}{\displaystyle{\lvert \Omega^{\sharp} \rvert }} &\text{ on } \partial \Omega^{\sharp},
            \end{cases}
        \end{equation}
        Sending $\sigma \to 0$ in \eqref{eq_per_v_sigma_e_u_sigma} we have
        \begin{align*}
            \norma{u}_{L^1(\Omega)} &\leq \lVert{u^\asso}\rVert_{L^1(\Omega^\sharp)}.
        \end{align*}

        \item[\textbf{Step 3}] Now we remove the assumption on the regularity of $\Omega$.
        
        Let $\Omega$ be a bounded, open and Lipschitz set, and $u \in W^{1,\infty}(\Omega)$.
        Then there exists a sequence $\Set{\Omega_k} \subset \R^n$ of open set with $C^{2}$ boundary such that $\Omega \subset \Omega_k, \; \forall k \in \mathbb{N}$ (for istance you can mollify $\chi_{\Omega}$ and take a suitable superlevel set) and
        \[
        \abs{\Omega_k \,  \triangle \, \Omega} \to 0 \qquad \mathcal{H}^{n-1}(\partial \Omega_k) \to \mathcal{H}^{n-1}(\partial \Omega) \qquad \text{ for } k \to +\infty .
        \]
        Let $\tilde{u}$ be an extension of $u$ in $\R^n$ such that
        \[
        \tilde{u} \rvert_{\Omega} \equiv u, \qquad \norma{\tilde{u}}_{W^{1,\infty}(\R^n)} \leq C \norma{u}_{W^{1,\infty}(\Omega)}.
        \]
        
        \noindent We define
        \[
        u_k = \tilde{u} \chi_{\Omega_k}, 
        \]
        
        and clearly $u_k = u$ in $\Omega$. By the previous step, we can construct $u_{k}^{\asso} \in W^{1,\infty}(\Omega_k^{\sharp})$ such that it is radial, $\abs{\nabla u_k}_* = \abs{\nabla u_k^{\asso}}_*$ and
        \begin{align}
            \label{confronto_norme_L^1_Lipschitz}
            \norma{u_k}_{L^1(\Omega_k)} &\leq \lVert u_k^{\asso} \rVert_{L^1(\Omega_k^{\sharp})} \\
            \label{tracce_uguali_Lipschitz}
            \int_{\partial \Omega_k} u_k \, d\mathcal{H}^{n-1}  & =\int_{\partial \Omega_k^{\sharp}} u_k^{\asso} \, d\mathcal{H}^{n-1}
        \end{align}
        Therefore, since $\lVert u_k \rVert_{W^{1,p}(\Omega_k)} \leq M$,  for all $p$, the  sequence $\Set{u_k^{\asso}}$ is equibounded in $W^{1,p}(\Omega^\sharp)$ and it has a subsequence which converges strongly in $L^p$ and weakly in $W^{1,p}$ to a function $w$.
        
        Let us prove that $\abs{\nabla u}$ and $\abs{\nabla w}$ has the same rearrangement.
        \[
            \limsup_k \, \bigl \lVert \abs{\nabla u_k^{\asso}} - \abs{\nabla u}_{\sharp} \bigr \rVert_{L^p(\Omega^{\sharp})} \leq \lim_k \, \bigl \lVert (f_k)_{\sharp} - f_{\sharp} \bigr \rVert_{L^p(\R^n)}
        \]
        where
        \[
            f (x) =
            \begin{cases}
                \abs{\nabla \tilde{u}} & \text{in }\Omega \\
                \lVert \nabla \tilde{u} \rVert_{L^{\infty}(\R^n)} & \text{in } \R^n \meno \Omega 
            \end{cases}
            \qquad \text{ and }
            f_k =
            \begin{cases}
                \abs{\nabla u_k} &\text{in } \Omega_k \\
                \lVert \nabla \tilde{u} \rVert_{L^{\infty}(\R^n)} & \text{in }\R^n \meno \Omega_k
            \end{cases}
        \]
        So using \eqref{eq_riarr_diminuiscono_distanza_Lp} we have
        \[
            \bigl \lVert (f_k)_{\sharp} - f_{\sharp} \bigr \rVert_{L^p(\R^n)} \leq \lVert f_k - f \rVert_{L^p(\R^n)} = \lVert f_k - f \rVert_{L^p(\Omega_k \meno \Omega)} \leq 2 \lVert \nabla \tilde{u} \rVert_{L^{\infty}(\R^n)} \abs{\Omega_k \meno \Omega}\
        \]
        that tends to $0$ as $k \to +\infty$ by the fact that $\abs{\Omega_k \triangle \Omega} \to 0$.

        \noindent Hence, the functions $\nabla w$ and $\nabla u$ has the same rearrangement, by the uniqueness of the weak limit in $\Omega^{\sharp}$.
        
        In the end, passing to limit $k \to +\infty$ in \eqref{confronto_norme_L^1_Lipschitz} and \eqref{tracce_uguali_Lipschitz}, we have
        \begin{align*}
            \lVert u \rVert_{L^1(\Omega)} &\leq \lVert w \rVert_{L^1(\Omega^{\sharp})} \\
            \int_{\partial \Omega} u \, d \mathcal{H}^{n-1} &= \int_{\partial \Omega^{\sharp}} w \, d\mathcal{H}^{n-1}.
        \end{align*}
        
        Hence $w= u^\asso$.

        \item[\textbf{Step 4}] Finally, we proceed by removing the assumption $u \in W^{1,\infty}(\Omega)$.
        
        If $u \in W^{1,p}(\Omega)$, by Meyers-Serrin Theorem, there exists a sequence $\{ u_k \} \subset C^{\infty}(\Omega) \cap W^{1,p}(\Omega)$ such that $u_k \to u$ in $W^{1,p}(\Omega)$. We can apply previous step to obtain $u_k^{\asso} \in W^{1,\infty}(\Omega^{\sharp})$ such that $\abs{\nabla u_k}$ and $\abs{\nabla u_k^{\asso}}$ are equally distributed and
        \begin{align}
            \label{confronto_norme_u_k}
            \norma{u_k}_{L^1(\Omega)} &\leq \norma{u_k^{\asso}}_{L^1(\Omega^{\sharp})} & & \forall k \in \mathbb{N} \\
            \label{uguaglianza_tracce_u_k}
            \int_{\partial \Omega} u_k \, d\mathcal{H}^{n-1} & = \int_{\partial \Omega^{\sharp}} u_k^{\asso} \, d\mathcal{H}^{n-1} & & \forall k \in \mathbb{N}.
        \end{align}
        Arguing as the previous step, there exists a function $w$ such that, up to a subsequence
        \[
            u_k^{\asso} \to w \text{ in } L^p(\Omega) \qquad \nabla u_k^{\asso} \convdeb \nabla w \text{ in } L^p(\Omega; \R^n)
        \]
        and $\abs{\nabla w}$ has the same rearrangement as $\abs{\nabla u}$.
        
        Finally, sending $k \to +\infty$ in \eqref{confronto_norme_u_k} and \eqref{uguaglianza_tracce_u_k}, we have
        \begin{align*}
            \norma{u}_{L^1(\Omega)} &\leq  \lVert w \rVert_{L^1(\Omega^{\sharp})} \\
            \int_{\partial \Omega} u \, d\mathcal{H}^{n-1} & = \int_{\partial \Omega^{\sharp}} w \, d\mathcal{H}^{n-1}.
        \end{align*}
        Hence $w= u^\asso$.\qedhere

    \end{enumerate}
\end{proof}

\section{An application to torsional rigidity}
\label{Section_4}
Let $\beta >0$, let $\Omega \subset \R^n$ be a bounded and  open set with Lipschitz boundary and let us consider the functional

\begin{equation}
    \mathcal{F}_{\beta}(\Omega, w) = \cfrac{ \displaystyle{  \int_{\Omega} \abs{\nabla w}^2 \, dx + \beta \abs{\partial \Omega}  \, \int_{\partial \Omega} w^2 \, d\mathcal{H}^{n-1} } }{\displaystyle{ \biggl(\int_{\Omega} w \, dx \biggr)^2 }} \qquad  w \in W^{1,2}(\Omega)
\end{equation}
and the associate minimum problem
\begin{equation}
    T(\Omega,\beta)  = \min_{ w \in W^{1,2}(\Omega) } \mathcal{F}_{\beta}(w)
\end{equation}
The minimum $u$ is a weak solution to
\begin{equation}
    \label{eq_soluzione_debole_Omega}
    \begin{cases} 
    -\Delta u = 1 &\text{in } \Omega \\[1ex]
    \displaystyle{\parzder{u}{\nu} + \beta \abs{\partial \Omega } \, u = 0} &\text{on } \partial \Omega
    \end{cases}
\end{equation}
Our aim is to compare $T(\Omega, \beta)$ with
\[
    T(\Omega^{\sharp},\beta) : = \min_{v \in W^{1,2}(\Omega)} \mathcal{F}_{\Omega,\beta}(v) = \min_{v \in W^{1,2}(\Omega)} \cfrac{ \displaystyle{ \int_{\Omega^\sharp} \abs{\nabla v}^2 \, dx +\beta \lvert \partial \Omega^{\sharp} \rvert \, \int_{\partial \Omega^\sharp} v^2 \, d\mathcal{H}^{n-1} } }{\displaystyle{ \biggl( \int_{\Omega^\sharp} v \, dx \biggr)^2 }}
\]
where the minimum is a weak solution to
\begin{equation}
    \label{eq_soluzione_debole_Omega_sharp}
    \begin{cases} 
    -\Delta z = 1 &\text{in } \Omega^{\sharp} \\[1ex]
    \displaystyle{\parzder{z}{\nu} + \beta \lvert \partial \Omega^{\sharp} \rvert \, z = 0} &\text{on } \partial \Omega^{\sharp}
    \end{cases}
\end{equation}

\begin{proof}[Proof of Corollary \ref{corollario_torsione_pesata}]
    Let $w \in W^{1,p}(\Omega)$, by Theorem \ref{Teorema_che_scriveremo} and Remark \ref{remark_holder} there exists $w^{\asso} \in W^{1,\infty}(\Omega^{\sharp})$ radial such that
    \[
        \int_{\Omega} \abs{\nabla w}^2 \, dx = \int_{\Omega^{\sharp}} \lvert \nabla w^{\asso} \rvert^2 \, dx \qquad \int_{\Omega}  \abs{w} \, dx \leq  \int_{\Omega^{\sharp}}  \lvert w^{\asso} \rvert \, dx \qquad \lvert \partial \Omega^{\sharp} \rvert  \,  \int_{\partial \Omega^{\sharp}} (w^\asso)^2 \leq  \abs{ \partial \Omega }  \, \int_{\partial \Omega} w^2
    \]
    Therefore
    \[
        \mathcal{F}_{\beta}(w) \geq \mathcal{F}_{\beta}(w^{\asso})
    \]
    Passing to the infimum on right-hand side and successively to the left-hand side, we obtain
    \[
        T(\Omega, \beta) \geq T(\Omega^{\sharp},\beta)
    \]
\end{proof}

\begin{oss}
    We highlight that all the arguments work also in the non-linear case, where the functional
    \begin{equation}
        \mathcal{F}_{\beta,p}(w) = \cfrac{ \displaystyle{ \int_{\Omega} \abs{\nabla w}^p \, dx + \beta \abs{\partial \Omega}^{p-1}   \, \int_{\partial \Omega} w^p \, d\mathcal{H}^{n-1} } }{\displaystyle{ \biggl(\int_{\Omega} w \, dx \biggr)^p }} \qquad \text{for } w \in W^{1,p}(\Omega).
    \end{equation}
    is considered.
\end{oss}

\vspace{1 em}

\section{\texorpdfstring{A weighted $L^1$ comparison}{}}
\label{Section_5}
Let us check how extend the result by \cite{Ta6} to the case of  function non vanishing on the boundary.

\begin{teorema}
    \label{Teorema_con_f}
    Let $\Omega \subset \R^n$ be a bounded, open and Lipschitz set. Let $f \in L^{\infty}(\Omega)$ be a function such that
    \begin{equation}
        \label{condizione_per_g}
        f^*(t) \geq \biggl( 1-\frac{1}{n} \biggr) \frac{1}{t} \int_0^t f^*(s) \, ds \qquad \forall t \in [0, \abs{\Omega}].
    \end{equation}
    If $u \in W^{1,p}(\Omega)$ and $u^{\asso}$ is the function given by Theorem \ref{Teorema_che_scriveremo}, then
    \begin{equation}
        \label{f-giiann}
        \int_{\Omega} f(x) u (x) \, dx \leq \int_{\Omega^{\sharp}} f^{\sharp} (x) u^{\asso} (x) \, dx.
    \end{equation}
\end{teorema}

\begin{proof}
    If $u \in W_0^{1,p}(\Omega)$, the result is contained in \cite{Ta6}. We recall it, for sake of completeness.

    By \cite[eq. 2.7]{GN} it is known
    \begin{equation}
        \label{Giarrusso_Nunziante_puntuale}
        u^*(s) \leq \frac{1}{n \omega_n^{\frac{1}{n}}} \int_s^{\abs{\Omega}} \frac{F(t)}{t^{1-\frac{1}{n}}} \, dt
    \end{equation}
    where $F$ is a function such that
    \[
    \int_0^s F(t) \, dt = \int_{D(s)} \abs{\nabla u}_{*}(s) \, ds
    \]
    with $D(s)$ defined in Section \ref{Section_2}.
    
    \noindent Setting $\displaystyle{g(t) := \frac{1}{t^{1-\frac{1}{n}}} \int_0^t f^*(s) \, ds}$, multiplying both terms of \eqref{Giarrusso_Nunziante_puntuale} for $f^*(s)$, integrating from $0$ to $\abs{\Omega}$ and using Fubini's Theorem we get
    \begin{equation}
        \label{integrale_f_u}
        \int_0^{\abs{\Omega}} f^*(s) u^*(s) \, ds \leq
        \frac{1}{n\omega_n^{\frac{1}{n}}} \int_0^{\abs{\Omega}} f^*(s) \biggl( \int_s^{\abs{\Omega}}  \frac{F(t)}{t^{1-\frac{1}{n}}} \, dt \biggr) \, ds
        = \frac{1}{n\omega_n^{\frac{1}{n}}} \int_0^{\abs{\Omega}} F(t) g(t)
        \, dt
    \end{equation}
    Let us suppose that $g(t)$ is non-decreasing, so $g_*(s) = g(s)$ and by Lemma \ref{lemma_Alvino_Trombetti} there exists a sequence $\{F_k \}$ such that $(F_k)_* = (\nabla u)_*$ and $F_k \convdeb F$ in $BV$. Therefore
    \[
        \int_0^{\abs{\Omega}} F(t) g(t) \, dt  = \lim_k \int_0^{\abs{\Omega}} F_k(t) g(t) \, dt
    \]
    Using Hardy-Littlewood's inequality we have
    \[
        \lim_k \int_0^{\abs{\Omega}} F_k(t) g(t) \, dt \leq \int_0^{\abs{\Omega}} \abs{\nabla u}_*(t) g_*(t) \, dt = \int_0^{\abs{\Omega}} \abs{\nabla u}_*(t) g(t) \, dt
    \]
    Hence, by \eqref{integrale_f_u} and Fubini's Theorem, we obtain
    \begin{align*}
        \int_0^{\abs{\Omega}} f^*(t) u^*(t) \, dt &\leq \frac{1}{n\omega_n^{\frac{1}{n}}} \int_0^{\abs{\Omega}} \abs{\nabla u}_*(t) \, g(t) \, dt \\
        & = \frac{1}{n\omega_n^{\frac{1}{n}}} \int_0^{\abs{\Omega}} \abs{\nabla u}_*(t) \Biggl( \frac{1}{t^{1-\frac{1}{n}}} \int_0^t f^*(s) \, ds \Biggr) \, dt \\
        & = \int_0^{\abs{\Omega}} f^*(s) \biggl( \frac{1}{n\omega_n^{\frac{1}{n}}} \int_s^{\abs{\Omega}} \frac{\abs{\nabla u}_*(t)}{t^{1-\frac{1}{n}}} \, dt \biggr) \, ds \\
        & = \int_0^{\abs{\Omega}} f^*(s) (u^{\asso})^*(s) \, ds
    \end{align*}
    Therefore, by Hardy-Littlewood inequality, we have
    \begin{equation}
        \int_{\Omega} f(x)u(x)  \, dx \leq \int_0^{\abs{\Omega}} f^*(t) u^*(t) \leq \int_0^{\abs{\Omega}} f^*(s) (u^{\asso})^*(s) \, ds = \int_{\Omega^{\sharp}} f^{\sharp}(x) \, u^{\asso}(x) \, dx
    \end{equation}
    But we have to deal with the assumption that $g$ is non-decreasing, that is
    \begin{equation*}
        g'(t) \geq 0  \iff \frac{d}{dt} \biggl( \frac{1}{t^{1-\frac{1}{n}}} \int_0^t     f^*(s) \, ds \biggr) = - \frac{n-1}{n} \frac{1}{t^{2-\frac{1}{n}}} \biggl(     \int_0^t f^*(s) \, ds \biggr) + \frac{1}{t^{1-\frac{1}{n}}}f^*(t) \geq 0,
    \end{equation*}
    hence, if and only if 
    \begin{equation*}
    	  f^*(t) \geq \biggl( 1-\frac{1}{n} \biggr) \frac{1}{t} \int_0^t f^*(s) \, ds.
    \end{equation*}

    \noindent Now let us deal with $u \notin W^{1,p}_0(\Omega)$. Suppose that $u \in C^2(\Omega)$ is a non-negative function, that $\Omega$ has $C^2$ boundary and that $f$ satisfies \eqref{condizione_per_g}. Proceeding as in Step 1 of Theorem \ref{Teorema_che_scriveremo}, for every $\varepsilon>0$ we can construct $u_{\varepsilon}$ that coincides with $u$ in $\Omega$ and is zero on $\partial \Omega_{\varepsilon}$. Moreover we can extend $f$ to $\Omega_{\varepsilon}$ simply defining
    \[
        f_{\eps} (t) =
        \begin{cases}
            f(x) &\text{in } \Omega \\
            f^*(\abs{\Omega}) &\text{in } \Omega_{\varepsilon} \meno \Omega
        \end{cases}
    \]
    The rearrangement, for every $\varepsilon>0$, is
    \[
        f_{\varepsilon}^*(t) =
        \begin{cases}
            f^*(t) &\text{in } \bigl[ 0,\abs{\Omega} \bigr ] \\[1ex]
            f^*(\abs{\Omega}) &\text{in } \bigl[  \abs{\Omega}, \abs{\Omega_{\varepsilon}} \bigr ],
        \end{cases}
    \]
    so we just have to check \eqref{condizione_per_g} for $t \in \bigl[ \abs{\Omega}, \abs{\Omega_{\varepsilon}} \bigr]$, namely
    \begin{equation}
        \label{27_fuori}
        f_{\varepsilon}^*(t)  \geq \biggl( \frac{n-1}{n} \biggr) \frac{1}{t} \int_0^{t} f_{\varepsilon}^*(s) \, ds.
    \end{equation}
    Keeping in mind that $f$ verifies \eqref{condizione_per_g}, we have
    \[
        f_{\varepsilon}^*(t)= f^{\ast}(\abs{\Omega}) \geq\biggl( \frac{n-1}{n} \biggr) \frac{1}{\abs{\Omega}} \int_0^{\abs{\Omega}} f^*(s) \, ds.
    \]
    If we show that 
    \[
        \frac{1}{\abs{\Omega}} \int_0^{\abs{\Omega}} f^*(s) \, ds \geq\left[  \frac{1}{t} \int_0^{\abs{\Omega}} f^*(s) \, ds + \frac{t-\abs{\Omega}}{t} f^*(\abs{\Omega}) \right] =\frac{1}{t} \int_0^{t} f_\varepsilon^*(s) \, ds
    \]
    then \eqref{27_fuori} is true.
    By direct calculations
    \[
        \frac{t-\abs{\Omega}}{t \abs{\Omega}} \int_0^{\abs{\Omega}} f^*(s) \, ds \geq \frac{t-\abs{\Omega}}{t} f^*(\abs{\Omega}) \iff \frac{1}{\abs{\Omega}} \int_0^{\abs{\Omega}} f^*(s) \, ds  \geq f^*(\abs{\Omega}).
    \]
    that is true of the fact that $f^*$ is decreasing.

    So, $\forall \varepsilon >0 $ we can apply the first part of the Theorem obtaining
    \[
        \int_{\Omega_{\varepsilon}} u_{\varepsilon} f_{\varepsilon} \, dx \leq \int_{\Omega_{\varepsilon}^{\sharp}}  v_{\varepsilon} f_{\varepsilon}^{\sharp} \, dx
    \]
    Sending $\varepsilon \to 0$ we get
    \[
        \int_{\Omega} u f \, dx \leq \int_{\Omega_{\sharp}} u^{\asso} f^{\sharp} \, dx.
    \]
    
    Arguing as in Theorem \ref{Teorema_che_scriveremo}, we get \eqref{f-giiann}. 
\end{proof}

\begin{oss}
    Condition \eqref{condizione_per_g} implies the $f$ is strictly positive.
    Moreover, if the essential oscillation of $f$ is bounded 
    \[ 
        \label{condizione_per_g_crescente}
        \essosc \abs{f} := \frac{\displaystyle{\esssup_{x \in \Omega} \abs{f(x)}}}{\displaystyle{\essinf_{x \in \Omega} \abs{f(x)}}} \leq \frac{n}{n-1}
    \]
    then \eqref{condizione_per_g} is satisfied.
\end{oss}

Theorem \ref{Teorema_con_f} allows us to compare the minimum of
\[
    T_{\beta, f}(\Omega) : = \min_{w \in W^{1,2}(\Omega)}\left\{\frac{1}{2} \int_{\Omega} \abs{\nabla w}^2 \, dx + \frac{\beta \abs{\partial \Omega}}{2}  \, \int_{\partial \Omega} w^2 \, d\mathcal{H}^{n-1} - \int_{\Omega} wf \, dx\right\}
\]
with the one of
\[
    T_{\beta, f}(\Omega^{\sharp}) : =\min_{v \in W^{1,2}(\Omega^{\sharp})}  \left\{\frac{1}{2} \int_{\Omega^{\sharp}} \abs{\nabla v}^2 \, dx + \frac{\beta \lvert \partial \Omega^{\sharp} \rvert }{2}  \, \int_{\partial \Omega^{\sharp}} v^2 \, d\mathcal{H}^{n-1} - \int_{\Omega^{\sharp}} vf^{\sharp} \, dx \right\}.
\]

\begin{corollario}
    Let $\beta>0$, let $\Omega \subset \R^n$ be a bounded, open and Lipschitz set. If $f$ satisfies \eqref{condizione_per_g}, then denoting with $\Omega^{\sharp}$ the ball centered at the origin with same measure as $\Omega$, it holds
    \[
        T_{\beta,f}(\Omega) \geq T_{\beta, f^{\sharp}} (\Omega^{\sharp})
    \]
\end{corollario}

Moreover we can use Theorem \ref{Teorema_con_f} to get a comparison between Lorentz norm of $u$ and $u^{\asso}$.
\begin{corollario}
    Let $1 \leq p \leq \frac{n}{n-1}$, under the assumption of Theorem \ref{Teorema_che_scriveremo} it holds
    \begin{equation}
        \label{eq_norme_Lorentz_L_p1}
        \norma{u}_{L^{p,1}(\Omega)} \leq \lVert u^{\asso} \rVert_{L^{p,1}(\Omega^{\sharp})}
    \end{equation}
    where $u^{\asso}$ is the function given by Theorem \ref{Teorema_che_scriveremo}
\end{corollario}
    
\begin{proof}
    Let us explicit the $L^{p,1}$ norm of $u$
    \[
        \norma{u}_{L^{p,1} (\Omega)} =  \int_0^{+\infty} t^{\frac{1}{p}-1} u^*(t) \, dt = \int_0^{+\infty} t^{-\frac{1}{p'}} u^*(t) \, dt
    \]
    Hence by Theorem \ref{Teorema_con_f}, it is sufficient that
    \begin{equation}
        \label{condiz_per_norma_lorentz}
        t^{-\frac{1}{p'}} - \frac{n-1}{n} \frac{1}{t} \int_0^t s^{-\frac{1}{p'}} \, ds \geq 0.
    \end{equation}
    If we compute
    \[
        \frac{1}{t} \int_0^t s^{-\frac{1}{p'}} \, ds =  \frac{1}{t} p \, t^{-\frac{1}{p'}+1} =  p \, t^{-\frac{1}{p'}},
    \]
    then we have
    \[
        t^{-\frac{1}{p'}} - \frac{n-1}{n} \frac{1}{t} \int_0^t s^{-\frac{1}{p'}} \, ds =t^{-\frac{1}{p'}} \biggl( 1-\frac{n-1}{n}p \biggr) \geq 0 \iff p \leq \frac{n}{n-1}
    \]
    so \eqref{condiz_per_norma_lorentz} is true and we can apply Theorem \ref{Teorema_con_f} obtaining
    \[
        \int_0^{+\infty} t^{-\frac{1}{p'}} u^*(t) \, dt \leq \int_0^{+\infty} t^{-\frac{1}{p'}} u^{\asso}(t) \, dt
    \]
    that is \eqref{eq_norme_Lorentz_L_p1}.
\end{proof}

\begin{oss}
    We emphasize that the bound $p \leq \frac{n}{n-1}$ is the best we can hope for Lorentz norm $L^{q,1}$. Indeed, if by absurd \eqref{eq_norme_Lorentz_L_p1} holds for $p>\frac{n}{n-1}$, by the embedding of $L^{p,q}$ spaces, $L^{q,1}(\Omega) \subseteq L^{q,q}(\Omega) = L^q(\Omega)$, which gives a contradiction.
\end{oss}

\vspace{1em}

\addcontentsline{toc}{chapter}{Bibliografia}

\printbibliography[heading=bibintoc, title={References}]

@article {GN,
    AUTHOR = {Giarrusso, E. and Nunziante, D.},
    TITLE = {Symmetrization in a class of first-order {H}amilton-{J}acobi
              equations},
    JOURNAL = {Nonlinear Anal.},
    FJOURNAL = {Nonlinear Analysis. Theory, Methods \& Applications. An
              International Multidisciplinary Journal},
    VOLUME = {8},
    YEAR = {1984},
    NUMBER = {4},
    PAGES = {289--299},
    ISSN = {0362-546X},
    MRCLASS = {35F30},
    MRNUMBER = {739660},
    MRREVIEWER = {Pierre-Louis Lions},
    DOI = {10.1016/0362-546X(84)90031-2},
}

@article{P,
	title={Torsional rigidity, principal frequency, electrostatic capacity and symmetrization},
	author={P{\'o}lya, G.},
	journal={Quarterly of Applied Mathematics},
	year={1948},
	volume={6},
	pages={267-277}
}

@Article{AT,
    Author = {A. {Alvino} and G. {Trombetti}},
    Title = {{Sulle migliori costanti di maggiorazione per una classe di equazioni ellittiche degeneri}},
    FJournal = {{Ricerche di Matematica}},
    Journal = {{Ric. Mat.}},
    ISSN = {0035-5038},
    Volume = {27},
    Pages = {413--428},
    Year = {1978},
    Publisher = {Springer, Milan; Universit\`a degli Studi di Napoli ``Federico II'', Naples},
    Language = {Italian},
    MSC2010 = {35J70 35B45 35J20 35J25},
    Zbl = {0403.35027}
}

@book {GT,
    AUTHOR = {Gilbarg, D. and Trudinger, N. S.},
     TITLE = {Elliptic partial differential equations of second order},
    SERIES = {Classics in Mathematics},
 PUBLISHER = {Springer-Verlag, Berlin},
     month = {2001},
     PAGES = {xiv+517},
      ISBN = {3-540-41160-7},
   MRCLASS = {35-02 (35Jxx)},
  MRNUMBER = {1814364},
}

@book{G,
    booktitle = {Minimal surfaces and functions of bounded variation},
    author = {Giusti, E..},
    address = {Canberra},
    isbn = {0708112943},
    language = {eng},
    lccn = {79306909},
    publisher = {Dept. of Pure Mathematics},
    series = {Notes on pure mathematics ; 10},
    title = {Minimal surfaces and functions of bounded variation / Enrico Giusti ; notes by Graham H. Williams.},
    year = {1977},
}

@inbook{Ta6,
    author = {  Talenti , G.},
    title = {On functions whose gradients have a prescribed rearrangement},
    booktitle = {Inequalities and Applications},
    chapter = {},
    pages = {559-571},
    doi = {10.1142/9789812798879_0047}
}

@article {ALT,
    AUTHOR = {Alvino, A. and Lions, P.-L. and Trombetti, G.},
     TITLE = {On optimization problems with prescribed rearrangements},
   JOURNAL = {Nonlinear Anal.},
  FJOURNAL = {Nonlinear Analysis. Theory, Methods \& Applications. An
              International Multidisciplinary Journal},
    VOLUME = {13},
      YEAR = {1989},
    NUMBER = {2},
     PAGES = {185--220},
      ISSN = {0362-546X},
   MRCLASS = {90C48 (26B35 49A27)},
  MRNUMBER = {979040},
MRREVIEWER = {Bruce D. Craven},
       DOI = {10.1016/0362-546X(89)90043-6},
}

@article {Cia,
    AUTHOR = {Cianchi, A.},
     TITLE = {On the {$L^q$} norm of functions having equidistributed
              gradients},
   JOURNAL = {Nonlinear Anal.},
  FJOURNAL = {Nonlinear Analysis. Theory, Methods \& Applications. An
              International Multidisciplinary Journal},
    VOLUME = {26},
      YEAR = {1996},
    NUMBER = {12},
     PAGES = {2007--2021},
      ISSN = {0362-546X},
   MRCLASS = {26D10 (35B99 46E30 46N20)},
  MRNUMBER = {1386130},
MRREVIEWER = {W. P. Ziemer},
       DOI = {10.1016/0362-546X(95)00052-W},
}

@book {HLP,
    AUTHOR = {Hardy, G. H. and Littlewood, J. E. and P\'{o}lya, G.},
     TITLE = {Inequalities},
    SERIES = {Cambridge Mathematical Library},
 PUBLISHER = {Cambridge University Press, Cambridge},
      YEAR = {1988},
     PAGES = {xii+324},
      ISBN = {0-521-35880-9},
   MRCLASS = {26Dxx (01A75)},
  MRNUMBER = {944909},
}

@article {Ta5,
    AUTHOR = {Talenti, G.},
    TITLE = {Elliptic equations and rearrangements},
    JOURNAL = {Ann. Scuola Norm. Sup. Pisa Cl. Sci. (4)},
    FJOURNAL = {Annali della Scuola Normale Superiore di Pisa. Classe di
              Scienze. Serie IV},
    VOLUME = {3},
    YEAR = {1976},
    NUMBER = {4},
    PAGES = {697--718},
    ISSN = {0391-173X},
    MRCLASS = {35J35 (35P15)},
    MRNUMBER = {601601},
    MRREVIEWER = {Vladimir A. Kondratiev},
}

@article{CP,
    author = {Chiti, G.},
    title = {Rearrangements of functions and convergence in orlicz spaces},
    journal = {Applicable Analysis},
    volume = {9},
    number = {1},
    pages = {23-27},
    year  = {1979},
    publisher = {Taylor & Francis},
    doi = {10.1080/00036817908839248},
}

@book{AFP,
  author = {Ambrosio, L. and Fusco, N. and Pallara, D.},
  publisher = {Oxford Mathematical Monographs},
  title = {{Functions of bounded variation and free discontinuity problems}},
  year = 2000
}

@article {Ferone_Posteraro,
    AUTHOR = {Ferone, V. and Posteraro, M. R.},
     TITLE = {Maximization on classes of functions with fixed rearrangement},
   JOURNAL = {Differential Integral Equations},
  FJOURNAL = {Differential and Integral Equations. An International Journal
              for Theory and Applications},
    VOLUME = {4},
      YEAR = {1991},
    NUMBER = {4},
     PAGES = {707--718},
}

@article {Ferone_Posteraro_Volpicelli,
    AUTHOR = {Ferone, V. and Posteraro, M. R. and Volpicelli, R.},
     TITLE = {An inequality concerning rearrangements of functions and
              {H}amilton-{J}acobi equations},
   JOURNAL = {Arch. Rational Mech. Anal.},
  FJOURNAL = {Archive for Rational Mechanics and Analysis},
    VOLUME = {125},
      YEAR = {1993},
    NUMBER = {3},
     PAGES = {257--269},
}

@article {ACNT_sharp_estimates,
    AUTHOR = {Alvino, A. and Chiacchio, F. and Nitsch, C. and Trombetti, C.},
     TITLE = {Sharp estimates for solutions to elliptic problems with mixed
              boundary conditions},
   JOURNAL = {J. Math. Pures Appl. (9)},
  FJOURNAL = {Journal de Math\'{e}matiques Pures et Appliqu\'{e}es. Neuvi\`eme S\'{e}rie},
    VOLUME = {152},
      YEAR = {2021},
     PAGES = {251--261},
}

@article{ANT,
	title={A Talenti comparison result for solutions to elliptic problems with Robin boundary conditions}, 
	author={Alvino, A. and Nitsch, C. and Trombetti, C.},
	year={2020},
	JOURNAL = {To appear on Comm. Pure Appl. Math.},
	FJOURNAL = {Communications on Pure and Applied Mathematics},
}

@article {AGM,
    AUTHOR = {Amato, V. and Gentile, A. and Masiello, A. L.},
     TITLE = {Comparison results for solutions to {$p$}-{L}aplace equations
              with {R}obin boundary conditions},
   JOURNAL = {Ann. Mat. Pura Appl. (4)},
  FJOURNAL = {Annali di Matematica Pura ed Applicata. Series IV},
    VOLUME = {201},
      YEAR = {2022},
    NUMBER = {3},
     PAGES = {1189--1212},
}
%\addcontentsline{toc}{chapter}{Bibliografia}
%\bibliographystyle{alpha}
%\bibliography{biblio}

\begin{abstract}
    \textsc{Dipartimento di Matematica e Applicazioni “R. Caccioppoli”, Universita` degli Studi di Napoli “Federico II”, Complesso Universitario Monte S. Angelo, via Cintia - 80126 Napoli, Italy.}
    
    \textsf{e-mail: vincenzo.amato@unina.it}

	\vspace{0.5cm}
	
	\textsc{Mathematical and Physical Sciences for Advanced Materials and Technologies, Scuola Superiore Meridionale, Largo San Marcellino 10, 80126 Napoli, Italy.}
	
	\textsf{e-mail: andrea.gentile2@unina.it}
\end{abstract}

\end{document}